\documentclass[a4paper, 12pt]{article}

\newif\ifAMS
\IfFileExists{amssymb.sty}
  {\AMStrue\usepackage{amssymb}}{}
\usepackage{latexsym}
\usepackage{tikz}
\usepackage{amsthm}
\usepackage{amsmath}

\usepackage{graphicx,textcomp,booktabs,amsmath,pst-node}
\usepackage{mathptmx,courier}
\usepackage[scaled]{helvet}

\newtheorem{thm}{Theorem}[section]
\newtheorem{cor}[thm]{Corollary}
\newtheorem{prop}{Proposition}[section]
\newtheorem{lem}[thm]{Lemma}
\newtheorem{rem}{Remark}[section]

\theoremstyle{definition}
\newtheorem{defn}{Definition}[section]

\numberwithin{equation}{section}

\usepackage{amsmath}

\usepackage[ansinew]{inputenc}
\usepackage[T1]{fontenc}
\usepackage{mathptmx,courier}
\usepackage[scaled]{helvet}
\usepackage[arrow, matrix, curve]{xy}
\usepackage{hyperref}
\begin{document}
\bigskip
\bigskip
\bigskip
\bigskip
\bigskip

\title{Maximal green sequences\\
for preprojective algebras}

\bigskip
\bigskip
\bigskip
\bigskip

\author{Magnus Engenhorst\footnote{engenhor@math.uni-bonn.de}\\
Mathematical Institute, University of Bonn\\
Endenicher Allee 60, 53115 Bonn, Germany\\[5mm]}

\maketitle

\begin{abstract}
Maximal green sequences were introduced as combinatorical counterpart for Donaldson-Thomas invariants for 2-acyclic quivers with potential by B. Keller. We take the categorical notion and introduce maximal green sequences for hearts of bounded t-structures of triangulated categories that can be tilted indefinitely. We study the case where the heart is the category of modules over the preprojective algebra of a quiver without loops. The combinatorical counterpart of maximal green sequences for Dynkin quivers are maximal chains in the Hasse quiver of basic support $\tau$-tilting modules. We show that a quiver has a maximal green sequence if and only if it is of Dynkin type. More generally, we study module categories for finite-dimensional algebras with finitely many bricks.           
\end{abstract}

\section{Introduction}

The motivation for this paper arose from the connection from maximal green sequences of a 2-acylic quiver $Q$ to stable modules over the Jacobi algebra $J(Q,W)$ for some non-degenerate potential $W$.\cite{10} Maximal green sequences were introduced by B. Keller in \cite{70} as certain sequences of mutations of 2-acyclic quivers $Q$ that correspond to sequences of simple tilts in the finite-dimensional derived catgeory of the Ginzburg algebra of $(Q,W)$ (cf. \cite{72}). We are interested in the categorical side of this correspondence and introduce maximal green sequences associated to Abelian subcategories of triangulated catgeories which are 'nice'. More precisely, a maximal green sequence is a certain sequence of simple tilts of a algebraic heart of a bounded t-structure of a triangulated category that we can tilt indefinitely (cf. Definition \ref{mgs}). An example is the category of finite-dimensional nilpotent modules $\mathcal{A}=\mathcal{P}(Q)-nil$ over the preprojective algebra $\mathcal{P}(Q)$ of a quiver $Q$ without loops inside the bounded derived category $\mathcal{D}^{b}(\mathcal{A})$. In the case of a Dynkin quiver we replace the derived category by a 'better behaved' 2-Calabi-Yau category described in \cite{1} but Theorem 1.1 also holds for the bounded derived category of $\mathcal{P}(Q)$ for a Dynkin quiver $Q$. If a quiver with potential has a maximal green sequence then its associated Jacobi algebra is finite-dimensional (Theorem 5.4 in \cite{70} and Prop. 8.1  in \cite{75}). We have the following analogue in our case (Propositions \ref{prop2} and \ref{prop3}):

\begin{thm}
Let $Q$ be a quiver without loops. Then the following is equivalent:
\begin{enumerate}
\item[(i)] There is a maximal green sequence of $\mathcal{P}(Q)$.
\item[(ii)] $Q$ is of Dynkin type.
\item[(iii)] $\mathcal{P}(Q)$ is finite-dimensional.
\end{enumerate}
\end{thm}

The equivalence of (ii) and (iii) is classical (see for instance \cite{137}). Stable modules over $\mathcal{P}(Q)$ play a key role in the proof of Theorem 1.1. They play a central role in the proof of the Kac conjecture for indivisible dimension vectors by W. Crawley-Boevey and M. van den Bergh in \cite{160}. Their key step in the formulation of Proposition \ref{cbb} is also crucial for this paper. Another motivation comes from work on spaces of stability conditions for preprojective algebras \cite{1,170}. Further, stable modules over preprojective algebras show up in the work of S. Cecotti on BPS states in \cite{3}. In the case of a Dynkin quiver $Q$ Y. Mizuno gave a identification of the Hasse quivers of torsion classes in $\mathcal{P}(Q)-nil$, basic support $\tau$-tilting modules and the Weyl group with the weak (Bruhat) order (Theorem 4.1 in \cite{5}). We show that the fact that we have finitely many bricks in this case underlies this classification. We give an explicit one-to-one correspondence between maximal green sequences for finite-dimensional algebras $A$ with finitely many bricks and torsion classes in $A-mod$ (Proposition \ref{prop4}).


Maximal green sequences are induced by discrete central charges on $\mathcal{P}(Q)-nil$ by Prop. 4.1 in \cite{10}. Enumerative results for reduced decompositions of the longest element of the Weyl group imply upper bounds on the number of (ordered) BPS spectra with respect to a discrete central charge (remark 4.1). This could be of interest in physics.\\ 

\textbf{Notation:} Given a set of objects or full subcategories $E_i$ for $i\in I$ for some index set $I$ $\left\langle E_i:i\in I\right\rangle$ will denote the extension-closed full subcategory generated by $E_i$ with $i\in I$.

\section{Preprojective algebras}

In this section we review results on preprojective algebras that will be used in the next sections. For more details see e.g. \cite{140} or \cite{155}.\\ 

Let $Q=(Q_0,Q_1)$ be a quiver without loops $\circlearrowleft$ and with $n$ vertices and let $h,t:Q_1\rightarrow Q_0$ be the head and tail maps. We obtain the double quiver $\overline{Q}$ from $Q$ by adding for every arrow $a:i\rightarrow j$ in $Q_1$ an arrow $a^*:j\rightarrow i$ in the opposite direction. Let $\mathbb{C}\overline{Q}$ be the patgh algebra of $\overline{Q}$.

\begin{defn}\cite{80}
The preprojective algebra $\mathcal{P}(Q)$ of $Q$ is defined by $$\mathcal{P}(Q):=\mathbb{C}\overline{Q}/(c)$$ where c is the ideal generated by $\sum_{a\in Q_{1}}(aa^*-a^*a)$.
\end{defn}

The preprojective algebra does not depend on the orientation of $Q$. Let $\mathcal{P}(Q)-mod$ be the category of finite-dimensional left $\mathcal{P}(Q)$-modules and $\mathcal{P}(Q)-nil$ the category of nilpotent finite-dimensional left $\mathcal{P}(Q)$-modules. A $\mathcal{P}(Q)$-module $M$ is nilpotent if a composition series of $M$ contains only the simple modules $S_1,\ldots, S_n$ associated to the $n$ vertices of $Q$. The category $\mathcal{P}(Q)-nil$ is of finite length with $n$ simple modules $S_1,\ldots, S_n$. If $Q$ is a Dynkin quiver the algebra $\mathcal{P}(Q)$ is finite-dimensional and all finite-dimensional $\mathcal{P}(Q)$-modules are nilpotent.\\


The modules in $\mathcal{P}(Q)-mod$ can be identified with the finite-dimensional representations $V=(V_i,\phi_a)$ of the quiver $\overline{Q}=(\overline{Q}_0,\overline{Q}_1)$ in which the linear maps $\phi_a, a\in\overline{Q}_1$ fulfill the relations $$\sum_{a\in Q_1:h(a)=i}\phi_{a}\phi_{a^*}-\sum_{a\in Q_1: t(a)=i}\phi_{a^*}\phi_{a}=0$$ for all $i\in Q_0$.\\

Let $(\ ,\ )$ be the symmetric bilinear form defined on the root lattice $$\mathbb{Z}Q_0=\mathbb{Z}[S_1]\oplus\cdots\oplus\mathbb{Z}[S_n]$$ by
\begin{align}
\label{cartan}
(x,y):=2\sum_{i\in Q_{0}}x_i y_i-\sum_{\substack{a:i\rightarrow j\\ a\in \overline{Q}_{1}}}x_i y_j.
\end{align}


\begin{prop}\cite{100}
\label{cb}
Let $Q$ be a quiver without loops and let $M,N$ be two finite-dimensional $\mathcal{P}(Q)$-modules. Then we have $$(\underline{dim}\ M, \underline{dim}\ N)=\mathrm{Hom}(M,N)+\mathrm{Hom}(N,M)-\dim\ Ext^1(M,N).$$ In particular, $Ext^1(M,N)=Ext^1(N,M)$. 
\end{prop}

A \textit{brick} is a module $M$ with $\mathrm{Hom}(M,M)=\mathbb{C}$. The following is well-known: 

\begin{lem}
\label{cecotti}
Let $Q$ be a Dynkin quiver and $M$ be a brick in $\mathcal{P}(Q)-mod$. Then $Ext^1(M,M)=0$ and $\underline{dim}\ M$ is a root, i.e. $(\underline{dim}\ M,\underline{dim}\ M)=2$.  
\end{lem}
\begin{proof}
Since $Q$ is Dynkin the symmetric bilinear form (\ref{cartan}) is positive definite. $(M,M)$ is even and therefore $Ext^1(M,M)$ vanishes by Proposition \ref{cb}. 
\end{proof}

Note that there are only finitely many bricks in $\mathcal{P}(Q)-mod$ for a Dynkin quiver $Q$ since all bricks $M$ are \textit{rigid}, i.e. $Ext^1(M,M)=0$ (cf. \cite{150}).

\section{Maximal green sequences}

We want to consider $\mathcal{P}(Q)-nil$ as an algebraic heart of a t-structure of a triangulated category $\mathcal{D}$ such that we can tilt indefinitely (cf. Definition \ref{indefinite}). For $Q$ not of Dynkin type we take $\mathcal{D}^b(\mathcal{P}(Q)-nil)$. B. Keller proved that $\mathcal{D}^b(\mathcal{P}(Q)-nil)$ has a Serre functor $[2]$, i.e. is a 2-Calabi-Yau category in this case \cite{110}. In the case of a Dynkin quiver $Q$ we replace the derived category by a better-behaved category $\mathcal{\hat{D}}$ described in \cite{1}: Let $G\subset SL_2(\mathbb{C})$ be a finite subgroup and let $\mathrm{Coh}_{G}(\mathbb{C}^2)$ denote the category of $G$-equivariant coherent sheaves on $\mathbb{C}^2$. Consider the full subcategory $\mathcal{A}\subset \mathrm{Coh}_{G}(\mathbb{C}^2)$ consisting of equivariant sheaves with no non-trivial $G$-equivariant sections. Then $\mathcal{\hat{D}}$ is the full subcategory of $\mathcal{D}^b(\mathrm{Coh}_{G}(\mathbb{C}^2))$ consisting of complexes whose cohomology sheaves lie in $\mathcal{A}$. The important fact for this paper is that $\mathcal{A}$ is equivalent to $\mathcal{P}(Q)-mod$ where $Q$ is a Dynkin quiver and $\mathcal{\hat{D}}$ is 2-Calabi-Yau.\\

Let $\mathcal{D}$ be the triangulated category $\mathcal{\hat{D}}$ described above in the case of a Dynkin quiver $Q$ and the bounded derived category $\mathcal{D}^b(\mathcal{P}(Q)-nil)$ else. Every simple module $S$ of $\mathcal{P}(Q)-nil$ is a 2-spherical object in $\mathcal{D}$, i.e. 
\begin{align*}
\mathrm{Hom}_{\mathcal{D}}^{i}(S,S)=\begin{cases}
\mathbb{C} & \text{if\ } i=0,2\\
0 & \text{else}
\end{cases}.
\end{align*} 

By \cite {120} every spherical object defines an auto-equivalence $\Phi_{S}$ of $\mathcal{D}$, the Seidel-Thomas twist, such that for every $E\in \mathcal{D}$ there is an exact triangle: 
\begin{align}
\label{seidelthomas}
\mathrm{Hom}^{\bullet}_{\mathcal{D}}(S,E)\otimes S\longrightarrow E\longrightarrow \Phi_{S}(E)\longrightarrow.
\end{align}
We identify throughout the Grothendieck groups $K(\mathcal{P}(Q)-nil)$ and $K(\mathcal{D})$ with the root lattice $\mathbb{Z}Q_0$. The induced linear map on the Grothendieck group $K(\mathcal{D})$ gives $$[\Phi_{S}(E)]=[E]-\chi(S,E)[S]$$ where $\chi:K(\mathcal{D})\times K(\mathcal{D})\rightarrow \mathbb{Z}$ is the Euler form

$$\chi(E,F)=\sum_{i\in\mathbb{Z}}(-1)^{i}\dim_{\mathbb{C}}\ \textrm{Hom}^{i}_{\mathcal{D}}(E,F).$$

We have
\begin{align*}
\chi(S_i,S_j)=\begin{cases}
2 & \text{if } i=j\\
\#(\text{arrows } i\rightarrow j\text{ in }Q)-\#(\text{arrows } j\rightarrow i\text{ in }Q) & \text{if } i\neq j 
\end{cases}
\end{align*}

and thus the lattice $(K(\mathcal{P}(Q)-nil),\chi(\ ,\ ))$ can be identified with the root lattice $(\mathbb{Z}Q_0, (\ ,\ ))$ associated to the quiver $Q$ \cite{158}.\\

Recall the notion of a bounded t-structure $\mathcal{C}\subset\mathcal{D}$ of a triangulated category $\mathcal{D}$ \cite{30}.

\begin{defn}
We call the heart $\mathcal{A}$ of a bounded t-structure of a triangulated category $\mathcal{D}$ \textit{algebraic} if 1. it has finite length, i.e. there are no infinite chains of inclusions or quotients for all objects and 2. it has finitely many simple objects. We call a heart $\mathcal{A}$ \textit{rigid} if all its simple objects $S$ are rigid, i.e. $Ext^1_{\mathcal{A}}(S,S)=0$.
\end{defn}

Note that an algebraic heart is a Krull-Schmidt category. Given a simple object $S$ in an algebraic heart $\mathcal{A}$ there is a well-known construction to define a new heart $\mathcal{A}_S$ of a bounded t-structure of $\mathcal{D}$, see \cite{40,50} We review it in the following.

\begin{defn}
Given two full subcategories $\mathcal{C}_{1}$ and $\mathcal{C}_{2}$ of an Abelian category $\mathcal{A}$ the \textit{Gabriel product} $\mathcal{C}_{1}\star\mathcal{C}_{2}$ is the full subcategory of objects $E$ that fit into a short exact sequence $$0\longrightarrow C_1\longrightarrow E\longrightarrow C_2 \longrightarrow 0$$ with $C_1\in\mathcal{C}_{1}$ and $C_2\in\mathcal{C}_{2}$.
\end{defn}

Then we have the following important definition:

\begin{defn}
A \textit{torsion pair} in an Abelian category $\mathcal{A}$ is a pair of full subcategories $(\mathcal{T}, \mathcal{F})$ satisfying
\begin{enumerate}
\item $\mathrm{Hom}_{\mathcal{A}}(T,F)=0$ for all $T\in\mathcal{T}$ and $F\in\mathcal{F}$;
\item every object $E\in\mathcal{A}$ is an element of the Gabriel product $\mathcal{T}\star\mathcal{F}$.
\end{enumerate} 
\end{defn}  

The objects of $\mathcal{T}$ are called \textit{torsion} and the objects of $\mathcal{F}$ are called \textit{torsion-free}, $\mathcal{T}$ is called \textit{torsion class} and $\mathcal{F}$ \textit{torsion-free class}.

\begin{prop} ~\cite{40}
\label{hrs}
Let $\mathcal{A}$ be the heart of a bounded t-structure on a triangulated category $\mathcal{D}$. Denote by $H^{i}(E)\in\mathcal{A}$ the i-th cohomology object of E with respect to this t-structure. Let $(\mathcal{T}, \mathcal{F})$ be a torsion pair in $\mathcal{A}$. Then the full subcategory
\begin{align}
\mathcal{A}^{*}=\left\{E\in \mathcal{D}| H^{i}(E)=0 \text{ for } i\notin \lbrace 0,1\rbrace, H^{0}(E)\in\mathcal{F}, H^{1}(E)\in \mathcal{T}\right\} \nonumber
\end{align}
is the heart of a bounded t-structure on $\mathcal{D}$. 
\end{prop}
We say $\mathcal{A}^{*}$ is obtained from $\mathcal{A}$ by \textit{(left) tilting} with respect to the torsion pair $(\mathcal{T}, \mathcal{F})$. The pair $(\mathcal{F}, \mathcal{T}[-1])$ is a torsion pair in $\mathcal{A}^{*}$.\\

Suppose $\mathcal{A}\subset\mathcal{D}$ is an algebraic heart of a bounded t-structure on $\mathcal{D}$. Given a simple object $S\in\mathcal{A}$ we can view $\left\langle S\right\rangle$ as the torsion class of a torsion pair on $\mathcal{A}$ with torsion-free class 
\begin{align}
\label{simpletilt}
\mathcal{F}=\left\{E\in\mathcal{A}|\mathrm{Hom}_{\mathcal{A}}(S,E)=0\right\}.
\end{align}
 This gives the \textit{simple (left) tilt of $\mathcal{A}$ at $S$}. If the heart $\mathcal{A}_S$ is again algebraic we can repeat this construction. The composition of left tilts is described by

\begin{lem}\cite{130}
\label{composition}
Let $\mathcal{A}$ be the heart of a bounded t-structure of a triangulated category. Let $(\mathcal{T}, \mathcal{F})$ be a torsion pair in $\mathcal{A}$ and $(\mathcal{T}', \mathcal{F}')$ a torsion pair in $\mathcal{A}^{*}=\left\langle\mathcal{F},\mathcal{T}[-1] \right\rangle$ If $\mathcal{T}'\subset \mathcal{F}$, then the left-tilt $\mathcal{A}^{**}=\left\langle\mathcal{F}',\mathcal{T}'[-1] \right\rangle$ of $\mathcal{A}^{*}$ equals the left-tilt of $\mathcal{A}$ with respect to the torsion pair $(\mathcal{T}\star\mathcal{T}', \mathcal{F}\cap\mathcal{F}')$.  
\end{lem}

\begin{defn}
\label{indefinite}
Let $\mathcal{A}$ be an algebraic heart of a bounded t-structure of a triangulated category $\mathcal{D}$ with $n$ simple objects. We say we can \textit{tilt $\mathcal{A}$ indefinitely} if any heart obtained from $\mathcal{A}$ by a finite sequence of simple tilts is again algebraic with $n$ simple objects. 
\end{defn}

Let $\mathcal{D}$ be the triangulated category $\mathcal{\hat{D}}$ described above in the case of a Dynkin quiver $Q$ and the derived category $\mathcal{D}^b(\mathcal{P}(Q)-nil)$ else. Then $\mathcal{A}=\mathcal{P}(Q)-nil$ is an algebraic heart of a t-structure in $\mathcal{D}$ and it is well-known that we can tilt $\mathcal{A}$ indefinitely, i.e. the hearts obtained by any finite sequence of simple tilts of $\mathcal{A}$ have finite length with $n$ simple objects. Further, these simple objects are again 2-spherical.

\begin{defn}
\label{mgs}
Let $\mathcal{A}$ be an algebraic heart of a bounded t-structure of a triangulated category $\mathcal{D}$ that we can tilt indefinitely. We call a finite sequence of simple tilts of the heart $\mathcal{A}$ such that we strictly tilt at objects in $\mathcal{A}$ a \textit{green sequence of $\mathcal{A}$}. If the last heart in the sequence is the shifted heart $\mathcal{A}[-1]$ we call it a \textit{maximal green sequence}. We call the number of simple tilts in a green sequence its \textit{length}.
\end{defn}

Note that by Lemma \ref{composition} all simple objects of a heart $\mathcal{A}'$ appearing in a green sequence lie in $\mathcal{A}$ or $\mathcal{A}[-1]$. Since the endomorphism rings of simple objects in $\mathcal{A}'$ are skew fields they are indecomposable in $\mathcal{A}$. The set of bounded t-structures of a triangulated category $\mathcal{D}$ forms a poset: For two bounded-t-structures $\mathcal{C}_1, \mathcal{C}_2\subset \mathcal{D}$ we have $$\mathcal{C}_1\leq \mathcal{C}_2\Leftrightarrow \mathcal{C}_1\subset\mathcal{C}_2.$$ 

\begin{lem}
\label{prop1}
Let $\mathcal{D}$ be a triangulated category and $\mathcal{A}$ an algebraic heart of a bounded t-structure of $\mathcal{D}$ with simple objects $S_1,\ldots, S_n$ that we can tilt indefinitely. Then we have the following:
\begin{enumerate}
\item[(i)] We tilt in a green sequence of $\mathcal{A}$ at an indecomposable object of $\mathcal{A}$ at most once.
\item[(ii)] We tilt in a maximal green sequence of $\mathcal{A}$ at all $n$ simple objects $S_1,\ldots, S_n$ of $\mathcal{A}$.  
\end{enumerate} 
\end{lem}
\begin{proof}
Ad (i). Note that there are no non-zero morphisms from $\mathcal{A}[i]$ to $\mathcal{A}[j]$ for $i>j$. If we tilt a heart at an indecomposable $S$ then $S[-1]$ remains in all following hearts in the green sequence since $\mathrm{Hom}(E,S[-1])=0$ for $E\in\mathcal{A}$ by (\ref{simpletilt}). Then the claim follows from the fact that we can not have $S$ and $S[-1]$ in only one heart.\\
Ad (ii). By Lemma \ref{composition} any heart $\mathcal{A}'$ appearing in a green sequence is given by the tilt at some torsion pair $(\mathcal{T},\mathcal{F})$ in $\mathcal{A}$, i.e. $\mathcal{A}'=\left\langle \mathcal{F},\mathcal{T}[-1]\right\rangle$. Thus we have for every object $C\in\mathcal{A}$ a short exact sequence $$0\longrightarrow A\longrightarrow C\longrightarrow B \longrightarrow 0$$ with $A\in\mathcal{T}$ and $B\in\mathcal{F}$. Thus the object $S_i$ or the object $S_i[-1]$ lie in $\mathcal{A}'$ for all $i$. There must be two hearts coming after each other in this sequence such that $\mathcal{A}'$ contains the simple $S_i$ for $i=1,\ldots,n$ and the consecutive heart $\mathcal{A}'_{S'}$ obtained from tilting $\mathcal{A}'$ at some simple object $S'$ of $\mathcal{A}'$ contains the object $S_i[-1]$. Thus in this case we have the short exact sequence
\begin{align*}
0\longrightarrow E\longrightarrow S_i[-1]\longrightarrow F\longrightarrow 0
\end{align*}
in $\mathcal{A}'_{S'}$ with $E\in\left\langle S' \right\rangle^{\bot}\subset\mathcal{A}'$ and $F\in \left\langle  S'\right\rangle[-1]$. The morphism $ S_i[-1]\rightarrow F$ is non-zero, otherwise $S_i[-1]$ would be in $\left\langle S' \right\rangle^{\bot}\subset\mathcal{A}'$. But we can not have $S_i$ and $S_i[-1]$ in $\mathcal{A}'$. Thus there is a non-zero morphism $f: S_i\rightarrow F'$ in $\mathcal{A}$ with $F'=F[1]\in\left\langle S' \right\rangle\subset\mathcal{A}$. Since $S_i$ is a simple object in $\mathcal{A}$ $f$ is injective with cokernel $coker\ f$. From the exact triangle $$S_i\longrightarrow F'\longrightarrow E[2]\longrightarrow$$ follows that $coker\ f\cong E[2]$. Thus $E\cong 0$ since $E$ is generated by objects in $\mathcal{A}$ and $\mathcal{A}[-1]$ and we have $S_i\cong S'$.
\end{proof}

\begin{prop}
\label{prop2}
Let $Q$ be a quiver of Dynkin type with $m$ positive roots. Then any green sequence can be completed to a maximal green sequence of $\mathcal{P}(Q)$. The length of any maximal green sequence is $m$ and we tilt at $m$ objects $E_1,\ldots E_m$ whose classes are the $m$ positive roots.
\end{prop}
\begin{proof}
Note that we have only finitely many bricks in $\mathcal{P}(Q)-nil$. By Lemma \ref{composition} all simple objects of hearts appearing in a green sequence are of the form $E$ or $E[-1]$ with $E$ a brick in $\mathcal{A}$. Now the first claim follows from Lemma \ref{prop1}(i).\\
It remains to show the second claim. For this we will anticipate notions from the next section. Let $\mathcal{D}$ be the triangulated category $\mathcal{\hat{D}}$ described above. In a maximal green sequence we tilt at a simple module say $S_1$ of $\mathcal{A}=\mathcal{P}(Q)-mod$ first. Given an algebraic heart $\mathcal{A'}$ in $\mathcal{D}$ with $n$ simple objects $S_1,\ldots, S_n$ we can define a Bridgeland stability condition $\sigma=(Z,\mathcal{P})$ by choosing a complex number in the upper half-plane $\overline{\mathbb{H}}$ for any simple $S_1,\ldots, S_n$ by Lemma 5.2 in \cite{50}. This defines a central charge $Z$ on the Grothendieck group $K(\mathcal{A}')=K(\mathcal{D})$ and the subcategories $\mathcal{P}(\phi)$ of $\sigma$-semistable objects of phase $\phi$ with $\phi\in(0,1]$ are exactly the semistable objects of $\mathcal{A}'$ with respect to $Z$ together with the zero objects.  Let us choose a central charge on $\mathcal{P}(Q)-mod$ such that the central charge $Z(S_1)$ is left to all central charges $Z(S_2),\ldots,Z(S_n)$ in the upper halfplane $\overline{\mathbb{H}}$. Note that all roots are indivisible since there are no imaginary roots for a Dynkin quiver. By Proposition \ref{ikeda} for any positive root $\alpha$ there is a semistable module with class $\alpha$. With the chosen central charge there are two possibilities: 1. The central charge of a semistable module $E$ with class $[E]=\alpha$ lies right to $Z(S_1)$ in the upper halfplane and $E$ is contained in the tilted heart $\mathcal{A}_{S}$ by Lemma \ref{wellknown} and (\ref{simpletilt}). 2. $Z(E)$ and $Z(S_1)$ have the same phase and thus $E=S_1$ since in this case $E\in \left\langle S_1\right\rangle $ and the class $[E]$ is indivisible. Note that there can not be an object in a heart with class $\alpha$ if there is already an object with class $-\alpha$. We tilt next at a simple object $S'$ of $\mathcal{A}_{S_{1}}$. We can again choose a central charge such that is $S'$ left-most. By Proposition \ref{ikeda} in the tilted heart $\mathcal{A}_{S_{1}}$ there is a semistable object $E'$ with class a positive root $\beta$. By the same arguments we have $S'=E'$ or $E'$ is contained in the next heart in the sequence. Going on in this way we see that we tilt at $m$ indecomposables with classes the positive roots. Since an object we tilt at in a maximal green sequence is a brick in $\mathcal{P}(Q)-mod$ and thus has class a positive root we tilt exactly at $m$ objects with classes the positive roots. 
\end{proof}


\begin{prop}
\label{prop3}
Let $Q$ be a quiver without loops. If there exists a maximal green sequence of $\mathcal{P}(Q)$ then $Q$ is of Dynkin type.
\end{prop}
\begin{proof}
This follows from the proof of Proposition \ref{prop2} since we have infinitely many indecomposable roots in the non-Dynkin case and thus have in any heart $\mathcal{A}'$ appearing in a green sequence of $\mathcal{P}(Q)$ infinitely many objects with class a positive indecomposable root.
\end{proof}

\begin{cor}
If $Q$ is a quiver without loops of non-Dynkin type, then there are infinitely many rigid bricks in $\mathcal{P}(Q)-nil$. 
\end{cor}
\begin{proof}
All indecomposables at that we tilt in a green sequence of $\mathcal{P}(Q)$ are rigid bricks. By Lemma \ref{prop1} and Proposition \ref{prop3} there have to be infinitely many rigid bricks.  
\end{proof}

By Lemma \ref{composition} a maximal green sequence of $\mathcal{P}(Q)$ for a Dynkin quiver $Q$ defines a sequence of torsion classes in $\mathcal{P}(Q)-mod$ ordered by inclusion. The set of torsion classes in $\mathcal{P}(Q)-mod$ with a relation given by inclusion is a partial ordered set. We give next a bijection between maximal chains in its Hasse quiver and maximal green sequences of $\mathcal{P}(Q)$. This will follow quickly from the results in \cite{135}.

\begin{defn}
We call an object $E$ in an algebraic heart $\mathcal{A}$ of a bounded t-structure of a triangulated category $\mathcal{D}$ \textit{endo-trivial}, if $\mathrm{Hom}_{\mathcal{A}}(E,E)$ is a skew field.
\end{defn}

An endo-trivial object is indecomposable. If $\mathcal{D}$ is a $k$-linear category over an algebraically closed field $k$, then $\mathrm{Hom}_{\mathcal{A}}(E,E)=k$ for an endo-trivial object $E$.
\vspace{0.2cm}\\
The following Lemma is a stronger version of Lemma 2.7 in \cite{135}:

\begin{lem}
\label{torsionclasses}
Let $\mathcal{A}$ be an algebraic heart of a bounded t-structure of a triangulated category $\mathcal{D}$ that we can tilt indefinitely. If the set of Grothendieck classes of endo-trivial objects in $\mathcal{A}$  is finite, then  $\mathcal{T}$ is of the form $(( \left\langle E_1\right\rangle \star \left\langle E_2\right\rangle )\star \cdots)\star \left\langle E_m\right\rangle $ where $E_1,\ldots, E_m$ are indecomposables at that we tilt in a green sequence of $\mathcal{A}$.
\end{lem}
\begin{proof}
The endomorphism ring of all simple objects of a heart that appears in a green sequence is a skew field. If there are only finitely many Grothendieck classes of indecomposables $E$ such that the endomorphism $\mathrm{Hom}_{\mathcal{A}}(E,E)$ is a skew field then any continued green sequence is a maximal green sequence by the proof of Lemma \ref{prop1}. Now the result follows similar to the proof of Lemma 2.7 in \cite{135}: Given a (non-trivial) torsion class $\mathcal{T}_0:=\mathcal{T}$ in $\mathcal{A}$ there is at least one simple object $E_1$ in $\mathcal{T}_0$ since torsion classes are closed under quotients and $\mathcal{A}$ is algebraic. Then $\mathcal{T}_1:=\mathcal{T}_{0}\cap \left\langle E_1\right\rangle^{\bot}$ is a torsion class in the tilted heart $\mathcal{A}'=\left\langle \left\langle E_1\right\rangle ^{\bot}, \left\langle E_1\right\rangle [-1] \right\rangle$ and in particular $\mathcal{T}_{1}\subset\mathcal{T}$. $\mathcal{T}_{1}$ is trivial or there is a non-zero epimorphism to a simple object $E_2$ of $\mathcal{A}'$ that is an element of $\mathcal{A}$. Going on in this way we will get a trivial torsion class $\mathcal{T}_{N}=0$ after a finite number $N$ of steps. By Lemma 2.7 in \cite{135} we get a torsion class $\left\langle E_1,E_2\ldots, E_{N-1}\right\rangle\subset\mathcal{T}_{0}$ in $\mathcal{A}$. We have $\left\langle E_1,E_2\ldots, E_{N-1}\right\rangle=\mathcal{T}_{0}$ otherwise there is an object $F\in\mathcal{T}_{0}$ but not in $\left\langle E_1,E_2\ldots, E_{N-1}\right\rangle$. Since $\left\langle E_1,E_2\ldots, E_{N-1}\right\rangle$ is a torsion class we have a short exact sequence $$0\longrightarrow A \longrightarrow F\longrightarrow B\longrightarrow 0$$ in $\mathcal{A}$ with $A\in\left\langle E_1,E_2\ldots, E_{N-1}\right\rangle$ and $B\in\left\langle E_1,E_2\ldots, E_{N-1}\right\rangle^{\bot}$. $B$ is an element of $\mathcal{T}_{0}$ since it is a quotient of $F\in\mathcal{T}_{0}$. Thus there is a non-zero object $B\in\mathcal{T}_{0}\cap\left\langle E_1,E_2\ldots, E_{N-1}\right\rangle^{\bot}=\mathcal{T}_{N}$. This is a contradiction. Applying Lemma \ref{composition} finishes the proof.
\end{proof}

E.g., the Grothendieck classes of bricks for the preprojective algebra of a Dynkin quiver are the finitely many positive roots.

\begin{cor}
\label{torsionclasses2}
Let $\mathcal{A}$ be an algebraic heart of a bounded t-structure of a triangulated category $\mathcal{D}$ that we can tilt indefinitely. If there are finitely many endo-trivial objects in $\mathcal{A}$, then there are finitely many torsion classes in $\mathcal{A}$. 
\end{cor}
\begin{proof}
By Lemma \ref{torsionclasses} any torsion class is of the form $\left\langle E_1,E_2\ldots\right\rangle$ for indecomposables $E_i$ with $\mathrm{Hom}_{\mathcal{A}}(E_i,E_i)$ a skew field. Thus there are finitely many torsion classes in $\mathcal{A}$.           
\end{proof}

\begin{rem}
A module category for a finite-dimensional algebra over a algebraically closed field $k$ with finitely many bricks, i.e. modules  $E$ with $\mathrm{Hom}(E,E)=k$, has finitely many torsion classes.
\end{rem}

Let $\mathfrak{H}^{\mathcal{A}}(\mathrm{tors)}$ be the Hasse quiver of the poset of torsion classes in the heart $\mathcal{A}$ of a bounded t-structure of a triangulated category $\mathcal{D}$ with relation given by inclusion: For two torsion classes $\mathcal{T}_{1}$ und $\mathcal{T}_{2}$ we have $\mathcal{T}_{1}\leq\mathcal{T}_{2}\Leftrightarrow \mathcal{T}_{1}\subset\mathcal{T}_{2}$. We can identify $\mathfrak{H}^{\mathcal{A}}(\mathrm{tors)}$ with the Hasse quiver of the poset of bounded t-structures $\mathcal{C}'\subset\mathcal{D}$ such $\mathcal{C}\leq\mathcal{C}'\leq\mathcal{C}[-1]$ where $\mathcal{C}$ is the t-structure associated to the heart $\mathcal{A}$ using the construction in Proposition \ref{hrs} by Proposition 2.3 in \cite{135}. We call these t-structures \textit{intermediate}. In the situation of Lemma \ref{torsionclasses} the hearts of all intermediate t-structures are algebraic. 

\begin{defn}
Let $\mathcal{A}$ be an algebraic heart of a bounded t-structure of a triangulated category $\mathcal{D}$ that we can tilt indefinitely. The \textit{exchange quiver} of $\mathcal{A}$ has as vertices the intermediate t-structures and we draw an arrow from the t-structure $\mathcal{C}_{1}$ to the t-structure $\mathcal{C}_{2}$ if we obtain $\mathcal{C}_{2}$ from $\mathcal{C}_{1}$ by a simple (left) tilt.
\end{defn}

\begin{prop}
\label{prop4}
Let $\mathcal{A}$ be an algebraic heart of a bounded t-structure of a triangulated category $\mathcal{D}$ that we can tilt indefinitely. If there are only finitely many Grothendieck classes of endo-trivial objects or torsion classes in $\mathcal{A}$, then $\mathfrak{H}^{\mathcal{A}}(\mathrm{tors)}$ coincides with the opposite quiver\footnote{The opposite quiver is the quiver with all arrows reversed.} of the \textit{exchange quiver} of $\mathcal{A}$.
\end{prop}
\begin{proof}
Let $\mathcal{A}_1$ and $\mathcal{A}_2$ be two hearts of intermediate t-structures such that $\mathcal{A}_2$ is obtained from $\mathcal{A}_1$ by a simple tilt at $S'\in\mathcal{A}_1$. There is some torsion pair $(\mathcal{T},\mathcal{F})$ in $\mathcal{A}$ with $\mathcal{A}_1=\left\langle \mathcal{F},\mathcal{T}[-1]\right\rangle $. We have $S'\in \mathcal{F}$ and thus $\mathcal{A}_2=\left\langle \mathcal{F}',\mathcal{T}'[-1]\right\rangle $ for the torsion class $\mathcal{T}'= \mathcal{T}\star\left\langle S'\right\rangle$ in $\mathcal{A}$ with $\mathcal{F}'=\mathcal{T}'^{\bot}$ by Lemma \ref{composition}. If there is a torsion class $\mathcal{T}''$ with $\mathcal{T}\subsetneq \mathcal{T}''\subset\mathcal{T}'$ then there is an object $t'\in\mathcal{T}''\setminus{{\mathcal{T}}}$ with $$0\longrightarrow t\longrightarrow t'\longrightarrow b\longrightarrow 0$$ where $t\in\mathcal{T}$ and $b\in\left\langle S'\right\rangle$. Thus there is an epimorphism $t\twoheadrightarrow S'$ and we have $S'\in\mathcal{T}''$ and $\mathcal{T}''=\mathcal{T}'$.\\
If there are only finitely many torsion classes then any continued green sequence is a maximal green sequence. It follows from the proof of Lemma \ref{torsionclasses} that there are indecomposables $E_1, \ldots ,E_m$ such that $\mathcal{T}=\left\langle E_1,\ldots, E_m\right\rangle $ for any torsion class $\mathcal{T}$ in $\mathcal{A}$. Let $\mathcal{T}_1$ and $\mathcal{T}_2$ be two torsion classes with $\mathcal{T}_1<\mathcal{T}_2$ in $\mathfrak{H}^{\mathcal{A}}(\mathrm{tors)}$. Then we have $\mathcal{T}_{1}=\left\langle E_1,\ldots, E_m\right\rangle $ and $\mathcal{T}_2=\left\langle E_1,\ldots, E_m,E_{m+1}\right\rangle $ for an indecomposable $E_{m+1}$ since  $\mathcal{T}_1\subset\mathcal{T}_2$ and there is no torsion class in between $\mathcal{T}_1$ and $\mathcal{T}_2$. Thus the hearts associated to the torsion classes $\mathcal{T}_1$ and $\mathcal{T}_2$ are related by a simple tilt at $E_{m+1}$ by Lemma 2.7 in \cite{135}.            
\end{proof}

In the case of $\mathcal{A}=\mathcal{P}(Q)-mod$ for a Dynkin quiver $Q$ we have the following classification of maximal green sequences of $\mathcal{P}(Q)$:

\begin{cor}
\label{cor2}
Let $Q$ be a Dynkin quiver. Then there are bijections between the following objects:
\begin{enumerate}
\item[(i)] The set of maximal green sequences of $\mathcal{P}(Q)$.
\item[(ii)] Maximal chains between the trivial element $e$ and the longest element $w_0$ in the Hasse quiver of the Weyl group $W_Q$ of $Q$ with respect to the weak right (Bruhat) order.
\item[(iii)] Maximal chains in the Hasse quiver between the trivial torsion class $0$ and the torsion class $\mathcal{T}=\mathcal{P}(Q)-mod$.    
\item[(iv)] Maximal chains in the Hasse quiver of basic support $\tau$ tilting modules bet-ween the modules $0$ and $\mathcal{P}(Q)$.   
\end{enumerate}
\end{cor}
\begin{proof}
The bijection between (i) and (iii) follows from Proposition \ref{prop4} and the bijection between (ii), (iii) and (iv) is Theorem 4.1 in \cite{5}.
\end{proof}

\begin{rem}
By the prefix and chain property of the weak order of $W_Q$ it follows that all maximal chains between the trivial element $e$ and the longest element $w_0$ have length $l(w_0)$. The length of the longest element $l(w_0)$ equals the number of reflections in $W_Q$ which are in bijection with the positive roots of $Q$ \cite{156,157}. This is in consistency with Proposition \ref{prop2}.  
\end{rem}

For the quiver $A_m$ there is an explicit formula of R. Stanley for the number of reduced decompositions of the longest element $w_0$ of $W_{A_{m}}$ that equals the numbers of maximal chains of $\mathcal{P}(A_m)$ in Corollary \ref{cor2} (ii) \cite{159}:

\begin{center}
\begin{tabular}{c|c|c|c|c|c}
quiver& $A_2$&$A_3$&$A_4$&$A_5$&$A_6$\\
\hline
\# maximal green sequences&2&16&2048&292864&1100742656\\
\end{tabular}. 
\end{center}


\section{Stable modules over preprojective algebras}

In this section we recall the notion of stability for Abelian categories and use it to construct examples of maximal green sequences for preprojective algebras. Further, we review results that are used in section 3 for the proof of the main result.

\begin{defn} A \textit{central charge} on an Abelian category $\mathcal{A}$ is a group Homomorphism $Z:K(\mathcal{A})\rightarrow\mathbb{C}$ such that for any nonzero $E\in\mathcal{A}$, $Z(E)$ lies in the upper halfplane
\begin{align}
\label{halfplane}
\overline{\mathbb{H}}:=\left\{r\cdot exp(i\pi\phi)| 0<\phi\leq 1,r\in\mathbb{R}_{>0}\right\}\subset\mathbb{C}.
\end{align}
\end{defn}

Every object $E\in\mathcal{A}$ has a phase $0<\phi(E)\leq 1$ such that $$Z(E)=r\cdot exp(i\pi\phi(E))$$ with $r\in\mathbb{R}_{>0}$. We say a nonzero object $E\in\mathcal{A}$ is \textit{(semi)stable} with respect to the central charge $Z$ if every proper subobject $0\neq A\subset E$ satisfies $\phi(A)< \phi(E)$ ($\phi(A)\leq \phi(E)$). A central charge is called \textit{discrete} if different stable object have different phases.\\

The following useful Lemma is well-known:

\begin{lem}
\label{wellknown}
Let $E$ and $F$ be semistable objects with respect to a central charge on an Abelian category $\mathcal{A}$. If we have $\phi(E)>\phi(F)$, then $\mathrm{Hom}_{\mathcal{A}}(E,F)=0$.
\end{lem}

We consider stable modules in $\mathcal{A}=\mathcal{P}(Q)-nil$. Let $\alpha$ be a class in $K(\mathcal{A})$. We call a central charge $Z:K(\mathcal{A})\rightarrow\mathbb{C}$ \textit{generic with respect to $\alpha$} if $\Im(Z(\beta)/Z(\alpha))\neq 0$ for all $0<\beta<\alpha$.\\
  
A root $\alpha$ in the root lattice of $Q$ is called \textit{indivisible} if there is no root $\beta$ with $\alpha=m\beta$ for an integer $m$ with $|m|>1$. The real roots $\alpha$ are indivisible since we have $(\alpha,\alpha)=2$ in this case. Further, every imaginary root is a multiple of an indivisible root and all non-zero multiples are roots (cf. \cite{158}).


\begin{prop}\cite{160}
\label{cbb}
Let $\alpha$ be a positive indivisible root and let $Z:K(\mathcal{P}(Q)-nil)\rightarrow\mathbb{C}$ be a generic central charge with respect to $\alpha$, then there is a stable module in $\mathcal{P}(Q)-nil$ with respect to $Z$ with class $\alpha$.
\end{prop}

For the definition of Bridgeland stability conditions on a triangulated category we refer to \cite{50}. Let $Stab^{\circ}(\mathcal{D})$ be the connected component of the space of stability conditions of the category $\mathcal{D}$ associated to a preprojective algebra containing the stability conditions with heart $\mathcal{A}=\mathcal{P}(Q)-nil$. Using the description of $Stab^{\circ}(\mathcal{D})$ given in \cite{1} for the Dynkin and affine case and in \cite{170} for the non-Dynkin case we can generalize this Proposition to any stability condition in $Stab^{\circ}(\mathcal{D})$:

\begin{prop}\cite{170}
\label{ikeda}
Given a stability condition $\sigma=(Z,\mathcal{P})$ in $Stab^{\circ}(\mathcal{D})$ and a indivisible root $\alpha$. Then there is a $\sigma$-semistable object with class $\alpha$. 
\end{prop}

The connection of stable modules to maximal green sequences is given by the following

\begin{prop}
\label{cmm}
Let $Q$ be a quiver without loops. Let $Z:K(\mathcal{P}(Q)-nil)\rightarrow\mathbb{C}$ be a discrete central charge with finitely many stable modules. Then the stable objects of $\mathcal{P}(Q)-nil$ in the order of decreasing phase define a maximal green sequence. 
\end{prop}
\begin{proof}
This follows immediately from the proof of Proposition 4.1 in \cite{10}.
\end{proof}

\begin{rem}
\label{rem1}
Note that Dynkin quivers automatically have finitely many stable modules since all stables are bricks. Thus Corollary \ref{cor2} gives a classification of (possible) BPS spectra ordered by phase.
\end{rem}

\section*{Acknowledgements} 

I thank Jan Schr\"oer and Tom Bridgeland for a helpful discussion respectively correspondence.

\end{document}